\documentclass{article}
\usepackage{amsmath}%
\usepackage{amsthm}%
\usepackage{amssymb}%

\usepackage{comment}

\newtheorem{theorem}{Theorem}[section]

\newtheorem{definition}[theorem]{Definition}
\newtheorem{proposition}[theorem]{Proposition}
\newtheorem{lemma}[theorem]{Lemma}
\newtheorem{remark}[theorem]{\it Remark}
\newtheorem{corollary}[theorem]{Corollary}
\newtheorem{example}[theorem]{\it Example}

\newtheorem{assumption}[theorem]{Assumption}

\renewcommand{\indent}{\hspace*{5mm}}

\newcommand{\cK}{{\cal K}}
\newcommand{\cL}{{\cal L}}
\newcommand{\cM}{{\cal M}}
\newcommand{\cN}{{\cal N}}

\newcommand{\cR}{{\cal R}}
\newcommand{\cS}{{\cal S}}

\begin{document}

\title{The law of the iterated logarithm in game-theoretic probability with quadratic and
stronger hedges}
\author{Kenshi Miyabe\thanks{Research Institute for Mathematical Sciences,  Kyoto University} \ and 
Akimichi Takemura\thanks{Graduate School of Information Science and Technology,  University of Tokyo}}

\date{\today}
\maketitle


\begin{abstract}
We prove both the validity and the sharpness of 
the law of the iterated logarithm in game-theoretic probability with quadratic and
stronger hedges.
\end{abstract}

\section{Background and the main result}

Assume that $\{X_n\}$ is a sequence of independent random variables
with $\mathbf{E}X_n=0$, $\mathbf{E}X_n^2<\infty$ and $n\ge1$.
Put $A_n=\sum_{i=1}^n\mathbf{E}X_i^2$ and $S_n=\sum_{i=1}^n X_i$.
The Kolmogorov law of the iterated logarithm (LIL) \cite{Kol29} says that
\begin{align}\label{eq:LIL}
\limsup_n \frac{S_n}{\sqrt{2A_n\ln\ln A_n}}=1\mbox{ a.s.}
\end{align}
if $A_n\to\infty$ for $n\to\infty$
and if there exists a sequence $\{c_n\}$ such that
\[|X_n|\le c_n=o(\sqrt{A_n/\ln\ln A_n})\mbox{ a.s.}\]
Our main result (Theorem \ref{th:main}) implies, as a corollary, that
the following is a sufficient condition for the LIL \eqref{eq:LIL}:
\begin{align}\label{eq:new-suff-cond}
A_n\to\infty\mbox{ and }\sum_n\frac{\mathbf{E}h(X_n)}{h(\sqrt{A_n/\ln\ln A_n})}<\infty
\end{align}
where $h$ satisfies Assumption \ref{assumption:1}.

We review some related results.
The restriction $|X_n|\le c_n$ of the Kolmogorov LIL is needed in a sense.
Marcinkiewicz and Zygmund \cite{MarZyg37} constructed
a sequence of independent random variables for which
$A_n\to\infty$ and $|X_n|=O(\sqrt{A_n/\ln\ln A_n})$
and which does not obey the LIL.
A number of other sufficient conditions for the LIL \eqref{eq:LIL}
were given in the literature
such as \cite{Ego90,Pet03}.
For instance, Egorov \cite{Ego84} showed that the following is a sufficient condition:
\[\frac{\sum_{i=1}^n X_i^2}{A_n}\to1\mbox{ a.s. }(n\to\infty)\mbox{ and }\]
\[\sum_{i=1}^n\mathbf{E}X_i^2 I(|X_i|>\frac{\epsilon A_n}{\ln\ln A_n})=o(A_n)\]
for any $\epsilon>0$.
Our result gives a new sufficient condition \eqref{eq:new-suff-cond}
for the LIL \eqref{eq:LIL}.
In the case of independent, identically distributed (i.i.d.)
random variables,
Hartman and Wintner \cite{HartWin41} proved that
existence of a second moment suffices for the LIL
and Strassen \cite{Str66} proved conversely that
existence of a second moment is necessary. 

The topic of this paper is the LIL in game-theoretic probability,
which was studied in Shafer and Vovk \cite{ShaVov01} under two protocols.
The first protocol ``unbounded forecasting'' only contains a quadratic hedge.

\begin{quote}
{\sc Unbounded Forecasting}\\
\textbf{Players}: Forecaster, Skeptic, Reality\\
\textbf{Protocol}:\\
\indent $\cK_0:=1$.\\
\indent FOR $n=1,2,\ldots$:\\
\indent\indent Forecaster announces $m_n\in\mathbb{R}$ and $v_n\ge0$.\\
\indent\indent Skeptic announces $M_n\in\mathbb{R}$ and $V_n\ge0$.\\
\indent\indent Reality announces $x_n\in\mathbb{R}$.\\
\indent\indent $\cK_n:=\cK_{n-1}+M_n(x_n-m_n)+V_n((x_n-m_n)^2-v_n)$.\\
\textbf{Collateral Duties}:
Skeptic must keep $\cK_n$ non-negative.
Reality must keep $\cK_n$ from tending to infinity.
\end{quote}

When Forecaster announces the range of $x_n$ at each round $n$, 
the game is called ``predictably unbounded forecasting''. 

\begin{quote}
{\sc Predictably Unbounded Forecasting}\\
\textbf{Players}: Forecaster, Skeptic, Reality\\
\textbf{Protocol}:\\
\indent $\cK_0:=1$.\\
\indent FOR $n=1,2,\ldots$:\\
\indent\indent Forecaster announces $m_n\in\mathbb{R}$, $c_n\ge0$, and $v_n\ge0$.\\
\indent\indent Skeptic announces $M_n\in\mathbb{R}$ and $V_n\in\mathbb{R}$.\\
\indent\indent Reality announces $x_n\in\mathbb{R}$ such that $|x_n-m_n|\le c_n$.\\
\indent\indent $\cK_n:=\cK_{n-1}+M_n(x_n-m_n)+V_n((x_n-m_n)^2-v_n)$.\\
\textbf{Collateral Duties}:
Skeptic must keep $\cK_n$ non-negative.
Reality must keep $\cK_n$ from tending to infinity.
\end{quote}

Let $A_n=\sum_{i=1}^n v_i$.
Shafer and Vovk \cite{ShaVov01} showed the following two theorems.

\begin{theorem}[Theorem 5.1 in \cite{ShaVov01}]\label{th:predict-lil}
In the predictably unbounded forecasting protocol,
Skeptic can force
\[\left(A_n\to\infty\ \&\ c_n=o\left(\sqrt{\frac{A_n}{\ln\ln A_n}}\right)\right)
\Longrightarrow\limsup_{n\to\infty}\frac{\sum_{i=1}^n(x_i-m_i)}{\sqrt{2A_n\ln\ln A_n}}=1.\]
\end{theorem}

\begin{theorem}[Theorem 5.2 in \cite{ShaVov01}]\label{th:unbound-restrict}
In the unbounded forecasting protocol,
Skeptic can force
\[\left(A_n\to\infty\ \&\ |x_n-m_n|=o\left(\sqrt{\frac{A_n}{\ln\ln A_n}}\right)\right)
\Longrightarrow\limsup_{n\to\infty}\frac{\sum_{i=1}^n(x_i-m_i)}{\sqrt{2A_n\ln\ln A_n}}\le1.\]
\end{theorem}

In the unbounded forecasting protocol,
it seems difficult to give a natural sufficient condition to force 
the lower bound of the LIL (cf.\ Proposition 5.1 of \cite{ShaVov01}).
Then we would like to find a non-predictable protocol under which 
a natural sufficient condition for the LIL exists.
A clue can be found in Takazawa \cite{Takazawa, takazawa-aism}
where he has showed a weaker upper bound with double hedges.
Another clue is the original proof \cite{HartWin41} of the Hartman-Wintner LIL
that uses a delicate truncation (see also Petrov \cite{Pet95}).
Thus we consider a game with stronger hedges
large enough to do the truncation.




\begin{quote}
{\sc The Unbounded Forecasting Game With Quadratic and Stronger Hedges}
(UFQSH)\\
\textbf{Parameter}: $h:\mathbb{R}\to\mathbb{R}$\\
\textbf{Players}: Forecaster, Skeptic, Reality\\
\textbf{Protocol}:\\
\indent $\cK_0:=1$.\\
\indent FOR $n=1,2,\ldots$:\\
\indent\indent Forecaster announces $m_n\in\mathbb{R}$, $v_n\ge0$ and $w_n\ge0$.\\
\indent\indent Skeptic announces $M_n\in\mathbb{R}$, $V_n\in\mathbb{R}$
and $W_n\in\mathbb{R}$.\\
\indent\indent Reality announces $x_n\in\mathbb{R}$.\\
\indent\indent $\cK_n:=\cK_{n-1}+M_n(x_n-m_n)+V_n((x_n-m_n)^2-v_n)$\\
\indent\indent \hspace{30mm}$+W_n(h(x_n-m_n)-w_n)$.\\
\textbf{Collateral Duties}:
Skeptic must keep $\cK_n$ non-negative.
Reality must keep $\cK_n$ from tending to infinity.
Forecaster must keep the game coherent.
\end{quote}

For simplicity we only consider an extra hedge $h$ with the following conditions.

\begin{assumption}
\label{assumption:1} \  
\begin{enumerate}
\item $h$ is an even function.
\item $h\in C^2$ and $h(0)=h'(0)=h''(0)=0$.
\item $h''(x)$ is strictly increasing, unbounded and concave (upward convex)
for $x\ge0$.
\end{enumerate}
\end{assumption}

Let 
\[
S_n=\sum_{i=1}^nx_i, \quad b_n=\sqrt{\frac{A_n}{\ln\ln A_n}}.
\]

We state our main result. 
\begin{theorem}\label{th:main}
In UFQSH with $h$ satisfying Assumption \ref{assumption:1},
Skeptic can force
\begin{align*}
\left(A_n\to\infty\mbox{ and }
\sum_n\frac{w_n}{h(b_n)}<\infty\right)
\Rightarrow
\limsup_{n\to\infty}\frac{S_n-\sum_{i=1}^n m_i}{\sqrt{2A_n\ln\ln A_n}}=1.
\end{align*}
\end{theorem}

This theorem is a consequence of Proposition \ref{th:upper-bound}
(upper bound, validity) and Proposition \ref{th:lower-bound} (lower bound, sharpness) below.
This theorem has the following corollary.

\begin{corollary}\label{cor:HW-S}
Let $h$ be an extra hedge satisfying Assumption \ref{assumption:1}
and
\begin{align}\label{eq:cond-h}
\sum_n\frac{1}{h(\sqrt{n/\ln\ln n})}<\infty.
\end{align}
In UFQSH with this $h$ 
and $m_n\equiv m$, $v_n\equiv v$ and $w_n\equiv w$,
the following are equivalent for $m'\in\mathbb{R}$ and $v'\ge0$.
\begin{enumerate}
\item $m'=m$ and $v'=v$.
\item Skeptic can force
\begin{align}\label{eq:HW-LIL}
\limsup_{n\to\infty}\frac{S_n-m'n}{\sqrt{2n\ln\ln n}}=\sqrt{v'}.
\end{align}
\item Reality can comply with \eqref{eq:HW-LIL}.
\end{enumerate}
\end{corollary}

The definition of ``comply'' is given in Definition \ref{def:comply}.

\begin{remark}
The equation \eqref{eq:HW-LIL} can be replaced with
\[\liminf_{n\to\infty}\frac{S_n-m'n}{\sqrt{2n\ln\ln n}}=-\sqrt{v'}.\]
\end{remark}

Examples for $h$ in this case are 
$h(x)=|x|^\alpha$, $2<\alpha\le 3$,
and $h(x)=(x+1)^2\ln^2(x+1)- x^2$.
See Example \ref{ex:power} and Example \ref{ex:example-hedge} below.

Our results have the following significance.
A game-theoretic version of Kolmogorov's LIL was established
by Shafer and Vovk \cite{ShaVov01},
in which a game-theoretic version of Hartman-Wintner's LIL was questioned.
As we stated, Takazawa \cite{Takazawa, takazawa-aism} also obtained
some related results.
Our main result gives a sufficient condition for
game-theoretic Kolmogorov's LIL
with an extra hedge slightly stronger than the quadratic one.
The corollary has a similar form as Hartman-Wintner's LIL
and Strassen's converse
although stronger hedges are assumed in our case.

\section{Facts and proofs}

In this section we give a proof of our main theorem and its corollary.
For readability our proof is divided into several sections.
We also prove some facts of independent interest.

\subsection{Consequences of the assumptions on the extra hedge}

From now on we assume $m_n \equiv 0$ without loss of generality
until Section \ref{subsec:proof-cor}.

\begin{proposition}\label{pro:h}
Under Assumption \ref{assumption:1}, we have
\begin{enumerate}
\item $\lim_{x\to 0}\frac{h'(x)}{x}=0$ and $\lim_{x\to0}\frac{h(x)}{x^2}=0$.
\item $\frac{h'(x)}{x}$ is strictly increasing and unbounded for $x\ge0$.
\item For $0\le c\le 1$ and for $x\ge0$ we have
\[c^3h(x)\le h(cx)\le c^2h(x).\]
For $c\ge 1$ and for $x\ge0$
\[c^2h(x)\le h(cx)\le c^3h(x).\]
\item 
$x^2=o(h(x))$.
\item $h(x)=O(x^3)$.
\item For any $b>0$, $\max_{y\ge 0} (1+y + y^2/2  - h(by)/h(b)) < 2$.
\end{enumerate}
\end{proposition}

\begin{proof}
(i)
Since $h''(0)=0$ and $h''$ is continuous, for each $\epsilon>0$,
there exists $\delta>0$ such that
\[h''(x)\le \epsilon\mbox{ for }0\le x\le\delta.\]
Then
\[h'(x)=\int_0^x h''(t)dt
\le\int_0^x \epsilon dt
=\epsilon x.\]
Thus $\lim_{x\to0}h'(x)/x=0$.
By a similar way, we can show that $\lim_{x\to0}\frac{h(x)}{x^2}=0$.

\bigskip
\noindent(ii)
The strict monotonicity of $h'(x)/x$ is equivalent to that,
for $y> 0$,
\begin{align*}
\frac{h'(x+y)}{x+y} >  \frac{h'(x)}{x}
\iff&xh'(x+y)-(x+y)h'(x) > 0\\
\iff&x(h'(x+y)-h'(x))>  y h'(x)\\
\iff&x\int_x^{x+y}h''(t)dt> y \int_0^xh''(t)dt.
\end{align*}
The last inequality holds because
\begin{align*}
x\int_x^{x+y}h''(t)dt
>  xy h''(x)
>  y\int_0^x h''(t)dt.
\end{align*}

We prove that $h'(x)/x$ is unbounded.
Since $h''$ is increasing and unbounded,
for any $C>0$, there exists $D>0$ such that
\[h''(x)>C\mbox{ for }x>D.\]
Then
\[h'(x)-h'(D)=\int_D^x h''(t)dt\ge \int_D^x Cdt=C(x-D)\]
for $x\ge D$.
Note that $C$ is arbitrary.


\bigskip
\noindent(iii)
We prove that $h(cx)\ge c^3h(x)$ for $c\le 1$.
By the concavity of $h''$, we have
\[h''(cx)\ge ch''(x).\]
Thus
\[
h'(cx)
=\int_0^{cx}h''(t)dt
=\int_0^x ch''(cs)ds
\ge \int_0^x c^2h''(s)ds
=c^2h'(x).
\]
Hence
\[
h(cx)
=
\int_0^{cx}h'(t)dt
=\int_0^x ch'(cs)ds\\
\ge
\int_0^x c^3h'(s)ds
=c^3h(x).
\]

Next we prove that $h(cx)\le c^2h(x)$ for $c\le 1$.
Since $h''$ is increasing, $h'$ is convex, thus
\[h'(cx)\le ch'(x)+(1-c)h'(0).\]
Then
\begin{align*}
h(cx)=\int_0^{cx}h'(t)dt
=\int_0^x ch'(cs)ds
\le c^2h'(x).
\end{align*}

The case of $c\ge 1$ is obtained from the first case by replacing $c$ and $cx$ by
$1/c$ and $x$, respectively.

\bigskip
\noindent(iv)
%
By the proof of (ii),
for any $C>0$, there exists $D>0$ such that
\[h'(x)\ge C(x-D)+h'(D)\]
for $x>D$.
Then
\begin{align*}
h(x)-h(D)
=&\int_D^x h'(t)dt
\ge\int_D^x(h'(D)+C(t-D))dt\\
=&h'(D)(x-D)+\frac{C(x^2-D)}{2}-CD(x-D).
\end{align*}
Since $C$ is arbitrary, $x^2=o(h(x))$.

\bigskip
\noindent(v)
By the inequality of (iv), for $x\ge1$,
\[\left(\frac{1}{x}\right)^3h(x)\le h\left(\frac{1}{x}\cdot x\right)=h(1).\]
Then $h(x)\le h(1)x^3$ for $x\ge1$.

\bigskip
\noindent(vi) Writing $y=c$ and $b=x$, by (iii)  for any $b>0$ we have
\[
\frac{h(by)}{h(b)} \ge\min(y^2,y^3) = \begin{cases} y^3 & \text{if}\  0 < y \le 1\\
                               y^2 & \text{if}\   y>1.
                             \end{cases}
\]
Hence 
\[
1+y+\frac{y^2}{2} - \frac{h(by)}{h(b)}\le 1 + y + \frac{y^2}{2} - \min(y^2, y^3).
\]
It is easy to  check numerically that the maximum of the right-hand side is less than 2.
\end{proof}

\begin{proposition}
In UFQSH, Skeptic's move should satisfy $W_n\ge0$ for each $n$.
\end{proposition}

\begin{proof}
Suppose that $W_n<0$ for some $n$.
It suffices to show that Reality can announce $x_n$ such that 
\[\cK_n=\cK_{n-1}+M_n x_n +V_n (x_n^2-v_n)+W_n(h(x_n)-w_n)<0,\]
which is equivalent to
\[h(x_n)>w_n-\frac{1}{W_n}(\cK_{n-1}+M_n x_n +V_n (x_n^2-v_n)).\]
This follows from (iv) of Proposition \ref{pro:h}.
\end{proof}

\subsection{A generalized H\"older's inequality}

Recall that a  game is called coherent if Reality can make
the capital not to increase at any round.
Intuitively the coherence means existence of a probability measure
such that Reality moves as if her move is based on the measure.
If $h(x)=x^k$, then, by H\"older's inequality, we expect that the coherence
implies $v_n^{1/2}\le w_n^{1/k}$ for all $n$.
We give a similar inequality for a general hedge $h$,
which we will use later.

\begin{proposition}\label{prop:vw}
In UFQSH with $h$ satisfying Assumption \ref{assumption:1}, the game is coherent if and only if
$h(\sqrt{v_n})\le w_n$ for all $n$.
\end{proposition}

\begin{proof}
Consider
\[
g(x;M,V,W)=M x + V (x^2 - v_n) + W (h(x)-w_n).
\]
Since the case $v_n=0$ or $w_n=0$ is trivial, we assume
$v_n,w_n>0$.
If $W=0$, then $\min_{x=\pm\sqrt{v_n}}g(x;M,V,W)\le0$.
If $W<0$, then $g(x;M,V,W)<0$ for a sufficiently large $x$.
Thus, we assume $W>0$ in the following.
Then  $g(\pm\infty;M,V,W)=\infty$ and $g(x;M,V,W)$ attains minimum
with respect to $x$ for fixed $M,V,W$.
The game is not coherent if and only if 
\[
\sup_{M,V,W}\min_x g(x;M,V,W) > 0
\]
at some round $n$.
If $V\ge 0$, then putting $x=0$ we have
\[
g(0;M,V,W) =-Vv_n - W w_n < 0,
\]
thus we ignore this case.
Furthermore we can let $M=0$
because 
$V (x^2 - v_n) + W (h(x)-w_n)$ is an even function and for any $x_0 >0$
\[
\min_{x=\pm x_0} g(x;M,V,W)= -|M| x_0 + V (x_0^2 - v_n) + W (h(x_0)-w_n).
\]
Now write
\[
g(x;0,V,W)= W \times \big(h(x)-w_n - U(x^2-v_n)\big)=W f(x;U), 
\]
where $U=-V/W > 0$. The game is not coherent if and only if
\[
\sup_{U>0} \min_{x>0}  f(x;U)>0
\]
for some $n$.
For $x>0$ 
\[
f'(x;U)=h'(x) - 2U x =2x ( \frac{h'(x)}{2x} - U).
\]
Hence for given $U$, the solution  $x=x(U)$ of $f'(x)=0$ 
is uniquely given by 
\begin{equation}
\label{eq:coherence-unique}
U=\frac{h'(x)}{2x}
\end{equation}
and $f$ takes the unique minimum at $x=x(U)$.
Now the right-hand side of \eqref{eq:coherence-unique} is strictly increasing in  $x$. Hence 
$x(U)$ is strictly increasing in $U$.  By the assumption on $h$, 
$x=x(U)$ is differentiable in $U$. Also note
$x(0)=0, x(\infty)=\infty$.
Let  
\[
\tilde f(U)=f(x(U);U)= h(x(U)) - w_n - U (x(U)^2 - v_n).
\]
We now maximize $\tilde f(U)$.
Differentiating $\tilde f(U)$ we have
\begin{align*}
\tilde f'(U)&= h'(x(U)) x'(U) - U \times (2 x(U) x'(U)) - (x(U)^2 - v_n) \\
&= [ h'(x(U)) - 2 U x(U) ] x'(U) - (x(U)^2 - v_n) \\
&= -  (x(U)^2 - v_n) .
\end{align*}
This implies that $\tilde f$ takes the unique maximum at $U=U^*$
satisfying $x(U^*)^2=v_n$.  By substituting $x(U^*)^2=v_n$
we have
\[
\max_{U>0} \min_{x>0} f(x;U) = \tilde f(U^*)=
h(x(U^*)) - w_n - U^* (x(U^*)^2 - v_n)= h(\sqrt{v_n})-w_n.
\]
Hence the game is not coherent if and only if $h(\sqrt{v_n})-w_n>0$ for some $n$.
\end{proof}

\subsection{Examples of the stronger hedge}


We give concrete examples of the stronger hedge
satisfying the conditions in Corollary \ref{cor:HW-S}.

\begin{example}
\label{ex:power}
Let $h(x)=|x|^\alpha$ for $2<\alpha \le 3$.
Then $h$ satisfies Assumption \ref{assumption:1} and
the condition \eqref{eq:cond-h}.
\end{example}


\begin{example}
\label{ex:example-hedge}
More elaborate example is the following hedge:
\begin{equation*}
h(x)= (1+x)^2 \ln^2(1+x)-x^2.
\end{equation*}
Note that $h(x)=x^2 \ln^2 x(1+o(x))$ as $x\rightarrow\infty$ and 
\[
\sum_{n} \frac{1}{h(\sqrt{n/\ln\ln n})} < \infty.
\]
This follows from the fact that for large $C$ the  following integral converges:
\[
\int_{C}^\infty \frac{1}{(x/\ln\ln x) \ln^2(x/\ln\ln x)}dx < \infty.
\]
Differentiating $h(x)$ successively we have
\begin{align*}
h'(x)&=2(1+x)\ln^2(1+x) + 2(1+x)\ln(1+x) - 2x,\\
h''(x)&=2 \ln^2(1+x) + 6 \ln(1+x),\\
h'''(x)&= \frac{4\ln(1+x)}{1+x} + \frac{6}{1+x},\\
h''''(x)&= -\frac{4\ln (1+x)}{(1+x)^2} - \frac{2}{(1+x)^2}.
\end{align*}
Hence $h\in C^2$, $h(0)=h'(0)=h''(0)=0$ and $h''$ is strictly increasing, unbounded and concave.
\end{example}

\subsection{Upper bound (validity)}

We show the upper bound of the LIL under our assumptions.

\begin{proposition}\label{th:upper-bound}
In UFQSH with $h$ satisfying Assumption \ref{assumption:1},
Skeptic can force
\begin{align}\label{eq:upper-bound}
\left(A_n\to\infty\mbox{ and }
\sum_n\frac{w_n}{h(b_n)}<\infty\right)
\Rightarrow
\limsup_{n\to\infty}\frac{S_n}{\sqrt{2A_n\ln\ln A_n}}\le1.
\end{align}
\end{proposition}

By Theorem \ref{th:unbound-restrict},
it suffices to show the following lemma.

\begin{lemma}\label{lem:x-bound}
In UFQSH with $h$ satisfying Assumption \ref{assumption:1},
Skeptic can force
\begin{align}\label{eq:x-bound}
A_n\to\infty\mbox{ and }\sum_n\frac{w_n}{h(b_n)}<\infty
\Rightarrow |x_n|=o(b_n) .
\end{align}
\end{lemma}

\begin{proof}
We consider the strategy with
\[\cK_0=D,\ M_n=V_n=0,\ 
W_n=\frac{1}{h(\epsilon b_n)}\]
as long as Skeptic can keep $\cK_n$ non-negative
where $\epsilon>0$ is small and $D$ is sufficiently large.
More precisely, we adopt a strategy combining accounts starting with $D=1,2,3,\ldots$
as in Miyabe and Takemura \cite{Miyabe_convrandseries}.
We show that this strategy forces \eqref{eq:x-bound}.

The capital process is
\[\cK_n=D+\sum_{i=1}^n\frac{h(x_i)}{h(\epsilon b_i)}
-\sum_{i=1}^n\frac{w_i}{h(\epsilon b_i)}.\]
By Proposition \ref{pro:h}, we have
\[h(\epsilon b_i)\ge \epsilon^3 h(b_i)\]
for all $i$.
Then
\begin{align*}
\cK_n
\ge&\cK_0+\sum_{i:|x_i|\ge\epsilon b_i}\frac{h(x_i)}{h(\epsilon b_i)}
-\sum_{i=1}^n\frac{w_i}{h(\epsilon b_i)}\\
\ge&\cK_0+\#\{1\le i\le n:|x_i|\ge\epsilon b_i\}-\frac{1}{\epsilon^3}\sum_{i=1}^n\frac{w_i}{h(b_i)}.
\end{align*}
For a large $D$, the strategy keeps $\cK_n$ non-negative.
Hence Skeptic can force that
\[\#\{1\le i\le n:|x_i|\ge\epsilon^3 b_i\}\]
is finite for each $\epsilon$.
\end{proof}

\subsection{Lower bound (sharpness)}

Next we show the lower bound of the LIL under the same assumptions.

\begin{proposition}\label{th:lower-bound}
In UFQSH with $h$ satisfying Assumption \ref{assumption:1},
Skeptic can force
\begin{align}\label{eq:lower-bound}
\left(A_n\to\infty\mbox{ and }
\sum_n \frac{w_n}{h(b_n)}<\infty\right)
\Rightarrow
\limsup_{n\to\infty}\frac{S_n}{\sqrt{2A_n\ln\ln A_n}}\ge1.
\end{align}
\end{proposition}

For our proof of the lower bound we closely follow the line
of argument in Section 5.3 of Shafer and Vovk \cite{ShaVov01}.
Compared to Section 5.3 of Shafer and Vovk \cite{ShaVov01}
we will explicitly consider rounds before appropriate stopping times.
Also we will be more explicit in choosing $\epsilon$'s and $\delta$'s.
 
We assume that a sufficiently small $\epsilon>0$ is chosen first and fixed.  For definiteness we let 
$\epsilon<1/8$.
We choose $\epsilon^*=\epsilon^*(\epsilon)>0$ sufficiently small compared to $\epsilon$, 
choose $\delta=\delta(\epsilon,\epsilon^*)>0$ sufficiently small,  and finally 
choose $C=C(\epsilon,\epsilon^*, \delta)>0$ sufficiently large.  

More explicitly, i) $\epsilon^*$ has to satisfy
\eqref{eq:epsilon-epsilonstar}
below, 
ii) $\delta$ has to satisfy 
\eqref{eq:22delta},
\eqref{eq:2delta},
\eqref{eq:epsilon-delta-log1},
\eqref{eq:epsilon-delta-log2},
\eqref{eq:delta-small-1},
\eqref{eq:delta-log-3},
\eqref{eq:delta-C-together-1},
\eqref{eq:epsilon-epsilonstar},
\eqref{eq:epsilon-delta-nn-positve-2}
below,  
and iii) $C$ has to satisfy
\eqref{eq:C-large-1},
\eqref{eq:delta-C-together-1},
\eqref{eq:C-delta},
\eqref{eq:tau3-positive}
below.

Let $\kappa$ be such that
\[
\kappa\le\sqrt{\frac{2\ln\ln C}{C}}.\]
Define stopping time $\tau_1,\tau_2,\tau_3$ 
by
\begin{align*}
&\tau_1=\min\left\{n\mid v_n>\delta^2\frac{C}{\ln\ln C}, w_n>\delta h\left(\sqrt{\frac{C}{\ln\ln C}}\right)
\right. \\
& \qquad \qquad \qquad \left. \mbox{ or } 
\sum_{i=1}^n w_i>\delta h\left(\sqrt{\frac{C}{\ln\ln C}}\right)\ln\ln C \right\},\\
&\tau_2=\min\{n\mid A_n\ge C\},\\
&\tau_3=\min\left\{n\mid |x_n|>\delta\sqrt{\frac{C}{\ln\ln C}}\right\}.
\end{align*}

In the following, we use a capital process that may be negative,
which is not allowed by the collateral duties,
in order to construct a non-negative capital process.
When Skeptic 
is allowed to sell tickets at the same price at which he can buy them, 
we say that the protocol is \emph{symmetric}.
The game of UFQSH is symmetric.
We call a capital process for Skeptic in a symmetric protocol
a (game-theoretic) \emph{martingale}.

\subsubsection{Approximations}

\begin{lemma}\label{lem:L-upperbound}
In UFQSH with $h$ satisfying Assumption \ref{assumption:1},
there exists a martingale $\cL_n=\cL_n^{\le,\kappa}$ such that $\cL(\Box)=1$
and
\begin{align}\label{eq:L-upperbound}
\frac{\cL_n}{\exp(\kappa \cS_n-\kappa^2C/2)}\le(\ln C)^{4\delta}
\end{align}
for $n$ such that $n=\tau_2<\tau_1,\tau_3$.
Furthermore $\cL_n$ is positive and
\begin{align}\label{eq:L-upperbound-small-n}
\frac{\cL_n}{\exp(\kappa \cS_n-(1-\delta)\kappa^2 A_n/2)}\le(\ln C)^{4\delta}
\end{align}
for $n$ such that $n\le\tau_2$ and $n<\tau_1,\tau_3$.
\end{lemma}

\begin{proof}
Consider the martingale $\cL$ satisfying $\cL(\Box)=1$ and
\[\cL_{i}=\cL_{i-1}\frac{1+\kappa x_i+\frac{\kappa^2 x_i^2}{2}-\frac{h(x_i)}{h(\kappa^{-1})}}
{1+\frac{\kappa^2v_i}{2}-\frac{w_i}{h(\kappa^{-1})}}\]
for all $i$.

We show that $\cL_n$ is positive for $n<\tau_1,\tau_3$.
First we prove that
\[1+\frac{\kappa^2v_i}{2}-\frac{w_i}{h(\kappa^{-1})}>0.\]
Note that
\[\frac{1}{\sqrt{2}}\cdot\sqrt{\frac{C}{\ln\ln C}}\le\kappa^{-1}.\]
Then
\begin{align}\label{eq:h-kappa}
h(\kappa^{-1})
\ge h\left(\frac{1}{\sqrt{2}}\cdot\sqrt{\frac{C}{\ln\ln C}}\right)
\ge \frac{1}{2\sqrt{2}}h\left(\sqrt{\frac{C}{\ln\ln C}}\right).
\end{align}
For $i<\tau_1$, we have
\[w_i\le \delta h\left(\sqrt{\frac{C}{\ln\ln C}}\right).\]
Then
\[\delta h(\kappa^{-1})\ge\frac{\delta}{2\sqrt{2}} h\left(\sqrt{\frac{C}{\ln\ln C}}\right)
\ge \frac{w_i}{2\sqrt{2}}\]
and
\begin{align}\label{eq:22delta}
\frac{w_i}{h(\kappa^{-1})}\le2\sqrt{2}\delta<1.
\end{align}
Hence
\[1+\frac{\kappa^2v_i}{2}-\frac{w_i}{h(\kappa^{-1})}
>1-\frac{w_i}{h(\kappa^{-1})}>0.\]

Next we prove that
\begin{align*}
1+\kappa x_i+\frac{\kappa^2 x_i^2}{2}-\frac{h(x_i)}{h(\kappa^{-1})}>0
\end{align*}
for $i<\tau_1,\tau_3$.
For $i<\tau_3$, we have
\begin{align}\label{eq:2delta}
|\kappa x_i|
\le\sqrt{\frac{2C}{\ln\ln C}}\cdot\delta\sqrt{\frac{C}{\ln\ln C}}
=\sqrt{2}\delta
<1.
\end{align}
Then
\[h(x_i)=h(\kappa x_i\cdot\kappa^{-1})\le|\kappa x_i|^2h(\kappa^{-1})
\le2\delta^2 h(\kappa^{-1}).\]

Next we show the inequality \eqref{eq:L-upperbound-small-n} for this $\cL_n$.
We claim that
\begin{align}\label{eq:compare-ex}
1+\kappa x_i+\frac{\kappa^2x_i^2}{2}-\frac{h(x_i)}{h(\kappa^{-1})}\le e^{\kappa x_i}.
\end{align}
for all $i$.
If $\kappa x_i\ge0$, then this inequality clearly holds.
If $\kappa x_i\le-1$, then 
\[1+\kappa x_i\le0\]
and
\[h(x_i)=h(\kappa^{-1}\kappa x_i)\ge|\kappa x_i|^2h(\kappa^{-1}),\]
thus the left-hand side of \eqref{eq:compare-ex} is non-positive.
If $-1<\kappa x_i<0$, then
\[h(x_i)=h(\kappa^{-1}\kappa x_i)\ge|\kappa x_i|^3h(\kappa^{-1}),\]
thus
\[1+\kappa x_i+\frac{\kappa^2x_i^2}{2}-\frac{h(x_i)}{h(\kappa^{-1})}
\le1+\kappa x_i+\frac{\kappa^2x_i^2}{2}+\frac{\kappa^3x_i^3}{6}
\le e^{\kappa x_i}.\]
Then
\begin{align*}
\prod_{i=1}^n (1+\kappa x_i+\kappa^2x_i^2/2-h(x_i)/h(\kappa^{-1}))
\le\prod_{i=1}^ne^{\kappa x_i}=e^{\kappa S_n}.
\end{align*}

Note that
\begin{equation}
\label{eq:epsilon-delta-log1}
0\le t\le \delta\Rightarrow \ln(1+t)\ge(1-\delta)t
\end{equation}
for sufficiently small $\delta$
and
\[\frac{\kappa^2 v_i}{2}\le\frac{2\ln\ln C}{C}\cdot\delta^2\frac{C}{\ln\ln C}\frac{1}{2}
=\delta^2.\]
Then if
\[\frac{\kappa^2 v_i}{2}-w_i/h(\kappa^{-1})\ge0,\]
we have
\[\ln(1+\frac{\kappa^2 v_i}{2}-w_i/h(\kappa^{-1}))\ge(1-\delta)\frac{\kappa^2 v_i}{2}-(1-\delta)w_i/h(\kappa^{-1}).\]

Note that
\begin{equation}
\label{eq:epsilon-delta-log2}
0\le t\le \delta\Rightarrow \ln(1-t)\ge-(1+\delta)t
\end{equation}
for sufficiently small $\delta$
and
\[\frac{w_i}{h(\kappa^{-1})}\le 2\sqrt{2}\delta.\]
for $i\le n<\tau_1$ by the fact that
$w_i\le\delta h(\sqrt{\frac{C}{\ln\ln C}})$ for $i\le n<\tau_1$ and \eqref{eq:h-kappa}.
Thus, if
\[\frac{\kappa^2 v_i}{2}-w_i/h(\kappa^{-1})<0,\]
then
\[\ln(1+\frac{\kappa^2 v_i}{2}-w_i/h(\kappa^{-1}))\ge(1+2\sqrt{2}\delta)\frac{\kappa^2 v_i}{2}-(1+2\sqrt{2}\delta)w_i/h(\kappa^{-1}).\]

By combining them, we have
\[\ln(1+\frac{\kappa^2 v_i}{2}-w_i/h(\kappa^{-1}))\ge(1-\delta)\frac{\kappa^2 v_i}{2}-(1+2\sqrt{2}\delta)w_i/h(\kappa^{-1}).\]
Thus
\begin{align*}
\sum_{i=1}^n\ln(1+\frac{\kappa^2 v_i}{2}-w_i/h(\kappa^{-1}))
\ge\frac{(1-\delta)\kappa^2}{2}\sum_{i=1}^nv_i
-(1+2\sqrt{2}\delta)/h(\kappa^{-1})\sum_{i=1}^nw_i
\end{align*}
and
\[\ln \cL_n
\le\kappa S_n-\frac{(1-\delta)\kappa^2}{2}\sum_{i=1}^n v_i
+\frac{(1+2\sqrt{2}\delta)}{h(\kappa^{-1})}\sum_{i=1}^n w_i.\]
By the inequality \eqref{eq:h-kappa}, we have
\[h(\kappa^{-1})
\ge \frac{1}{2\sqrt{2}}h\left(\sqrt{\frac{C}{\ln\ln C}}\right).\]
Hence
\begin{align*}
\sum_{i=1}^n w_i
\le\delta h\left(\sqrt{\frac{C}{\ln\ln C}}\right)\ln\ln C
\le2\sqrt{2}\delta h(\kappa^{-1})\ln\ln C
\end{align*}
for $n<\tau_1$.
Thus
\begin{align*}
\ln \cL_n
&\le\kappa S_n-\frac{(1-\delta)\kappa^2}{2}\sum_{i=1}^n v_i
+2\sqrt{2}\delta(1+2\sqrt{2}\delta)\ln\ln C\\
&\le\kappa S_n-\frac{(1-\delta)\kappa^2}{2} A_n
+4\delta\ln\ln C
\end{align*}
for sufficiently small $\delta$ 
such that
\begin{equation}
\label{eq:delta-small-1}
2\sqrt{2} (1+2\sqrt{2}\delta) < 3.
\end{equation}
Hence \eqref{eq:L-upperbound-small-n} is proved.

The inequality above also implies \eqref{eq:L-upperbound}
because, for $n=\tau_2$,
\begin{align*}
\ln \cL_n -\kappa \cS_n+\frac{\kappa^2C}{2}
\le&\frac{\kappa^2C}{2}-\frac{(1-\delta)\kappa^2C}{2}
+2\sqrt{2}\delta(1+\delta)\ln \ln C\\
\le&\delta\ln\ln C+2\sqrt{2}\delta(1+\delta)\ln\ln C\\
<&4\delta\ln\ln C.
\end{align*}

\end{proof}

\begin{lemma}\label{lem:L-lowerbound}
In UFQSH with $h$ satisfying Assumption \ref{assumption:1},
there exists a positive martingale $\cL_n=\cL^{\ge,\kappa}$ such that $\cL(\Box)=1$,
\begin{equation}
\label{eq:L-lowerbound}
\frac{\cL_n}{\exp(\kappa \cS_n-\kappa^2C/2)}\ge(\ln C)^{-4\delta}
\end{equation}
for $n$ such that $n=\tau_2<\tau_1,\tau_3$.  Furthermore
\begin{equation}
\label{eq:L-lowerbound-small-n}
\frac{\cL_n}{\exp(\kappa \cS_n-(1+\delta)\kappa^2 A_n/2)}\ge 1.
\end{equation}
for $n$ such that $n\le\tau_2$ and $n<\tau_1,\tau_3$.
\end{lemma}

The proof is the same as Lemma 5.2 in Shafer and Vovk \cite{ShaVov01}, 
except that we also explicitly consider $n < \tau_2$.

\begin{proof}
Let
\[f(t)=1+t+(1+\delta)\frac{t^2}{2}\]
and consider the martingale $\cL$ satisfying $\cL(\Box)=1$ and
\[\cL_{i}=\cL_{i-1}\frac{1+\kappa x_i+(1+\delta)\kappa^2 x_i^2/2}{1+(1+\delta)\kappa^2v_i/2}
=\cL_{i-1} \frac{f(\kappa x_i)}{1+(1+\delta)\kappa^2v_i/2}\]
for all $i$.
For $i<\tau_3$,
\[|\kappa x_i|\le \sqrt{\frac{2\ln\ln C}{C}}\cdot\delta\sqrt{\frac{C}{\ln\ln C}}
= \sqrt{2}\delta.\]
Since
\begin{equation}
\label{eq:delta-log-3}
|t|\le \sqrt{2}\delta\Rightarrow 1+t+(1+\delta)\frac{t^2}{2}\ge e^t,
\end{equation}
for sufficiently small $\delta$
we have
\[\prod_{i=1}^n f(\kappa x_i)\ge\prod_{i=1}^n e^{\kappa x_i}=e^{\kappa S_n}.\]
Since $\ln(1+t)\le t$,
\[\sum_{i=1}^n\ln(1+(1+\delta)\frac{\kappa^2 v_i}{2})
\le(1+\delta)\sum_{i=1}^n\frac{\kappa^2v_i}{2}.\]
It follows that
\begin{align*}
\ln \cL_n
\ge\kappa S_n-(1+\delta)\frac{\kappa^2}{2}\sum_{i=1}^n v_i
=\kappa S_n-(1+\delta)\frac{\kappa^2}{2} A_n.
\end{align*}
Hence \eqref{eq:L-lowerbound-small-n} is proved.

The last inequality implies \eqref{eq:L-lowerbound}
because, for $n=\tau_2$,
\begin{align*}
\ln \cL_n -\kappa \cS_n+\frac{\kappa^2C}{2}
\ge&\frac{\kappa^2C}{2}-(1+\delta)\frac{\kappa^2}{2}\left(C+\delta^2\frac{C}{\ln\ln C}\right)\\
=&-\delta\frac{\kappa^2}{2}C-(1+\delta)\frac{\kappa^2}{2}\delta^2\frac{C}{\ln\ln C}\\
\ge&-\delta\ln\ln C-(1+\delta)\delta^2\\
\ge&-4\delta\ln\ln C
\end{align*}
for sufficiently large $C$ such that
\begin{equation}
\label{eq:C-large-1}
3\ln \ln C > \delta(1+\delta).
\end{equation}
\end{proof}

\subsubsection{Construction of a martingale}

\begin{lemma}\label{lem:N}
Choose $C$ sufficiently large for a given $\epsilon$. 
In UFQSH with $h$ satisfying Assumption \ref{assumption:1},
there exists a martingale $\cN$ such that
\begin{enumerate}
\item $\cN(\Box)=1$,
\item For $n$ such that $n=\tau_2<\tau_1,\tau_3$ and
\[S_n\le(1-\epsilon)\sqrt{2C\ln\ln C},\]
we have
\[\cN_n\ge1+\frac{1}{\ln C}\]
\item $\cN_n$ is positive for $n$ such that $n <\tau_1$, $n\le\tau_2$ and
$n\le\tau_3$.
\end{enumerate}
\end{lemma}

\begin{proof}
Choose $\epsilon^*$ and $\delta$ sufficiently small and 
$C$ sufficiently large.
Let
\[
\kappa_1 = (1-\epsilon) \sqrt{\frac{2\ln\ln C}{C}}, \quad
\kappa_2 = (1+\epsilon^*) \kappa_1, \quad
\kappa_3 = (1+\epsilon^*) \kappa_2.
\]
Define a  martingale $\cM_n$ by
\[\cM_n=3\cL_n^{\le,\kappa_2} -\cL_n^{\ge,\kappa_1} - \cL_n^{\ge,\kappa_3},\]
where $\cL_n^{\le,\kappa}$ is the martingale bounded from above in 
Lemma \ref{lem:L-upperbound} and  $\cL_n^{\ge,\kappa}$ is the martingale
bounded from below in Lemma \ref{lem:L-lowerbound}. Furthermore define
$\cN_n$ by
\[\cN_n=1+\frac{1-\cM_n}{\ln C}.\]
Since $\cM(\Box)=1$, $\cN(\Box)=1$.

First we prove that $\cM_n\le0$
for $n=\tau_2<\tau_1,\tau_3$ and $S_n\le(1-\epsilon)\sqrt{2C\ln\ln C}$.
The value $\cM_n$ is bounded from above by
\begin{align*}
\cM_n &\le 3\cL_n^{\le,\kappa_2} -\cL_n^{\ge,\kappa_1}\\
&\le 3\exp((1+\epsilon^*)\kappa_1\cS_n-(1+\epsilon^*)^2\kappa_1^2C/2)(\ln C)^{4\delta}\\
&\ \ -\exp(\kappa_1\cS_n-\kappa_1^2C/2)(\ln C)^{-4\delta}\\
&=\exp(\kappa_1\cS_n-\kappa_1^2C/2)(\ln C)^{-4\delta}\\
&\ \ \times(3\exp(\epsilon^*\kappa_1\cS_n-\epsilon^*(2+\epsilon^*)\kappa_1^2C/2)(\ln C)^{8\delta}-1).
\end{align*}
This is negative because 
\begin{align*}
\epsilon^*\kappa_1S_n-\epsilon^*(2+\epsilon^*)\kappa_1^2C/2
\le&\epsilon^*(1-\epsilon)^22\ln\ln C-\epsilon^*(2+\epsilon^*)(1-\epsilon)^2\ln\ln C\\
\le&-(\epsilon^*)^2(1-\epsilon)^2\ln\ln C\\
<&-\ln3-8\delta\ln\ln C
\end{align*}
for sufficiently small $\delta$ and sufficiently large $C$ such that
\begin{equation}
\label{eq:delta-C-together-1}
8\delta < \frac{1}{2} (\epsilon^*)^2(1-\epsilon)^2, \qquad 
\frac{1}{2} (\epsilon^*)^2(1-\epsilon)^2 \ln\ln C > \ln 3.
\end{equation}

Next we prove that $\cN_n$ is positive for $n$
such that $n<\tau_1$, $n\le\tau_2$ and $n\le\tau_3$.
First we consider the case that $n\le\tau_2$ and $n<\tau_1,\tau_3$.
We distinguish two cases depending on the value of $S_n$.
Consider the case that 
\[
S_n < \kappa_3 A_n + \frac{5\delta\ln\ln C}{\kappa_2\epsilon^*}.
\]
Then by
Lemma \ref{lem:L-upperbound}
\begin{align*}
\ln \cL_n^{\le,\kappa_2} 
&\le \kappa_2 S_n - (1-\delta)\frac{\kappa_2^2}{2} A_n + 4\delta\ln\ln C\\
&\le  \kappa_2 (\kappa_3 A_n + \frac{5\delta\ln\ln C}{\kappa_2\epsilon^*}) 
- (1-\delta)\frac{\kappa_2^2}{2} A_n + 4\delta\ln\ln C \\
&=\frac{\kappa_2^2}{2}A_n(2(1+\epsilon^*) - (1-\delta)) + \frac{5+4\epsilon^*}{\epsilon^*}\delta\ln\ln C \\
&= \frac{\ln\ln C}{C}A_n(1+\epsilon^*)^2 (1 + 2\epsilon^* +\delta)(1-\epsilon)^2 + \frac{5+4\epsilon^*}{\epsilon^*}\delta\ln\ln C.
\end{align*}
Note that
$A_{n-1}<C$ and $v_n\le\delta^2\frac{C}{\ln\ln C}$ 
for $n$ such that $n\le\tau_2$ and $n<\tau_1$,
thus
\[A_n=A_{n-1}+v_n\le (1+\frac{\delta^2}{\ln\ln C})C\le(1+\delta)C\]
for $C$ such that
\begin{equation}\label{eq:C-delta}
\ln\ln C\ge \delta.
\end{equation}
Then
\[\frac{\ln \cL_n^{\le,\kappa_2}}{\ln\ln C}
\le(1+\delta)(1+\epsilon^*)^2 (1 + 2\epsilon^* +\delta)(1-\epsilon)^2 + \frac{5+4\epsilon^*}{\epsilon^*}\delta.\]
We can assume that
\begin{equation}
\label{eq:epsilon-epsilonstar}
c_\epsilon:=(1+\delta)(1+\epsilon^*)^2 (1 + 2\epsilon^* +\delta)(1-\epsilon)^2 + \frac{5+4\epsilon^*}{\epsilon^*}\delta < 1.
\end{equation}
Then we have
\begin{equation}
\label{eq:nn-positive-1}
\frac{\cL_n^{\le,\kappa_2}}{\ln C} \le (\ln C)^{c_\epsilon-1}  \rightarrow 0 \qquad ( C\rightarrow\infty)
\end{equation}
and in this case $\cN_n$ is positive for large $C$.

Now consider the other case $S_n \ge \kappa_3 A_n+  5\delta\ln\ln C/(\kappa_2\epsilon^*)$.  Then
\begin{align*}
\ln\frac{\cL_n^{\le,\kappa_2}}{\cL_n^{\ge,\kappa_3}}&\le 
\kappa_2 S_n - (1-\delta)\frac{\kappa_2^2}{2} A_n + 4\delta\ln\ln C - (\kappa_3 S_n - (1+\delta) \frac{\kappa_3^2}{2} A_n)\\
&=(\kappa_2 - \kappa_3) S_n + \frac{A_n}{2} \big((1+\delta) \kappa_3^2 - (1-\delta) \kappa_2^2\big) + 4\delta\ln\ln C\\
&= -\epsilon^*\kappa_2S_n + \frac{\kappa_2^2}{2}A_n \big((1+\delta)(1+\epsilon^*)^2- (1-\delta))   + 4\delta\ln\ln C\\
&\le -\epsilon^*\big((1+\epsilon^*) \kappa_2^2 A_n +\frac{5\delta\ln\ln C}{\epsilon^*}\big)
\\ & \qquad 
+ \frac{\kappa_2^2 }{2}A_n \big((1+\delta)(1+\epsilon^*)^2- (1-\delta))   + 4\delta\ln\ln C\\
&=\frac{\kappa_2^2}{2} A_n \big(- 2\epsilon^*(1+\epsilon^*)+ (1+\epsilon^*)^2 -1  + \delta((1+\epsilon^*)^2  + 1))
-\delta\ln\ln C \\
&=\frac{\kappa_2^2}{2} A_n \big(- (\epsilon^*)^2 + \delta((1+\epsilon^*)^2  + 1)) -\delta\ln\ln C <0 
\end{align*}
for $\delta$ such that 
\begin{equation}
\label{eq:epsilon-delta-nn-positve-2}
- (\epsilon^*)^2 + \delta((1+\epsilon^*)^2  + 1))  < 0 .
\end{equation}
In this case
\begin{equation}
\label{eq:nn-positive-2}
\frac{\cL_n^{\le,\kappa_2}}{\cL_n^{\ge,\kappa_3}} \le (\ln C)^{-\delta} \rightarrow 0   \qquad (C\rightarrow\infty)
\end{equation}
and $\cN_n$ is positive for large $C$.

Hence at round $n$ such that $n\le\tau_2<\tau_1,\tau_3$, 
$\cN_n$ is positive for large $C$ in both cases.

We finally consider the case that
$n=\tau_3$, $n\le\tau_2$ and $n<\tau_1$.
The difficulty with the stopping time $\tau_3$ is that
it depends on Reality's move $x_n$,
thus it is after Skeptic uses the strategy
that Skeptic know whether $n=\tau_3$.
We need to make sure that $\cN_n$ is positive even if Reality has
chosen a very large $|x_n|$ at the round $n$.
By (vi) of  Proposition \ref{pro:h}
\[ 
1+\kappa x_i + \frac{\kappa^2 x_i^2}{2} - \frac{h(x_i)}{h(\kappa^{-1})}
= 1 + y + \frac{y^2}{2} - \frac{h(by)}{h(b)}
\qquad (y=\kappa x_i, b=\kappa^{-1}).
\]
Hence  for all $x_i$ and $\kappa>0$ 
\[
1+\kappa x_i + \frac{\kappa^2 x_i^2}{2} - \frac{h(x_i)}{h(\kappa^{-1})} \le 2
\]
and the relative growth of $\cL^{\le, \kappa_2}_n$ is bounded by $3$ from above.
Hence at $n=\tau_3 \le  \tau_1, \tau_2$
\[\cL^{\le, \kappa_2}_{n} \le 3 \cL^{\le, \kappa_2}_{n-1}.\]
Also for all $x_i$ and $\kappa>0$ 
\[
1+\kappa x_i + (1+\delta)\frac{\kappa^2 x_i^2}{2} > 1+\kappa x_i + \frac{\kappa^2 x_i^2}{2} 
\ge \frac{1}{2}
\]
Hence the relative growth of $\cL^{\ge, \kappa_3}_n$ 
is bounded by $1/3$ from below. Hence
at $n=\tau_3 \le  \tau_1, \tau_2$
\[
\frac{\cL_{n}^{\le,\kappa_2}}{\cL_{n}^{\ge,\kappa_3}}
\le 9 \times \frac{\cL_{n-1}^{\le,\kappa_2}}{\cL_{n-1}^{\ge,\kappa_3}}.
\]
Then $\cN_n$ is positive at $n=\tau_3 \le  \tau_1, \tau_2$
by choosing $C$ large enough in 
\eqref{eq:nn-positive-1} and \eqref{eq:nn-positive-2} such that
\begin{equation}
\label{eq:tau3-positive}
(\ln C)^{c_\epsilon-1}  < 1/3  \quad \text{and}\ \quad
(\ln C)^{-\delta} < 1/9.
\end{equation}
\end{proof}

\subsubsection{Strategy forcing the lower bound}
Here we discuss Skeptic's strategy forcing the lower bound in Proposition \ref{th:lower-bound}.
For each sufficiently small $\epsilon>0$,  
we want to construct a positive capital process $\cK_n$ 
such $\limsup_n \cK_n = \infty$ for any path 
satisfying  the antecedent in \eqref{eq:lower-bound} and
\begin{equation}
\label{eq:lower-bound-violation}
S_n \le  (1-2\epsilon) \sqrt{2A_n \ln\ln A_n}
\end{equation}
for all sufficiently large $A_n$.
We also assume that Skeptic is already employing a strategy forcing the 
upper bound in LIL for $-S_n$ with a small initial capital.  
Hence $S_n \ge -(1+\epsilon) \sqrt{2A_n \ln \ln A_n}$ for all sufficiently large $A_n$. 
For a path satisfying  the antecedent in \eqref{eq:lower-bound} and the inequality in 
\eqref{eq:lower-bound-violation}, at the round $n'$ with $A_{n'}=(D+1)A_n$ we have 
\[
S_{n'} \le  (1-2\epsilon) \sqrt{2(D+1)A_n \ln \ln (D+1)A_n}.
\]
Then 
\[
S_{n'}- S_n  \le  (1-2\epsilon) \sqrt{2(D+1)A_n \ln\ln (D+1)A_n} + (1+\epsilon) \sqrt{2A_n \ln \ln A_n}
\]
Let $D=1/\epsilon^4$. Recall that we assumed $\epsilon<1/8$ for definiteness. 
For this $D=1/\epsilon^4$  it is easily seen that for all sufficiently large $A_n$
we have
\begin{align*}
&(1-2\epsilon) \sqrt{2(D+1)A_n \ln\ln (D+1)A_n} + (1+\epsilon) \sqrt{2A_n \ln \ln A_n} \\
&\qquad\qquad  \le (1-\epsilon) \sqrt{2(D+1)A_n \ln \ln (D+1)A_n}
\end{align*}
and
\[
S_{n'}- S_n  \le  (1-\epsilon) \sqrt{2(D+1)A_n \ln \ln (D+1)A_n}.
\]
Now, if necessary,  we increase $D$ to  $D=\max(C,1/\epsilon^4)$, where $C$ is taken sufficiently large to 
satisfy requirements (\eqref{eq:C-large-1},
\eqref{eq:delta-C-together-1},
\eqref{eq:tau3-positive})
in the previous sections.


Now we consider the following strategy based on the strategy of 
Lemma \ref{lem:N} with $C$ replaced by $D^k$ where $k\in\mathbb{N}$.

\begin{quote}
Start with initial capital $\cK=1$.\\
Set $k=1$.\\
Do the followings repeatedly:\\
\indent $C:=D^k$.\\
\indent Apply the strategy in Lemma \ref{lem:N} until\\
\indent\indent(i) $v_n>\delta^2\frac{C}{\ln\ln C}$,\ 
$w_n>\delta h(\sqrt{\frac{C}{\ln\ln C}})$,\\
\indent\indent\indent\mbox{ or }$\sum_{i=1}^n w_i>\delta h\left(\sqrt{\frac{C}{\ln\ln C}}\right)\ln\ln C$,\\
\indent\indent(ii) $A_n\ge C$, \\
\indent\indent\mbox{ or }\\
\indent\indent(iii) $|x_n|>\delta \sqrt{C/\ln\ln C}$,\\
\indent Set $k=\max\{k+1,\min\{m\ :\ D^m>A_n\}\}$.
\end{quote}

The ``until'' command is understood
exclusively for (i),
but inclusively (ii) and (iii).
If (i) happens, Skeptic does not apply the strategy of Lemma \ref{lem:N}
and let $0=M_n=V_n=W_n$.  He
increases $k$ (and $C$)
so that (i) does not hold (such $k$ always exists)
and Skeptic can apply the strategy for the increased $C$.
If (ii) happens, Skeptic continues to apply the strategy
and go to the next $k$ after that.
Note that, Skeptic can observe whether (i) or (ii) happened or not 
before his move, because (i) and  (ii) only depend on Forecaster's move, but
he  knows whether (iii) happens or not
only after Skeptic applied a strategy,
so ``until'' command should be inclusive for (iii).
This point was already  discussed at the end of our proof of Lemma \ref{lem:N}.

Suppose that the path satisfies the antecedent in \eqref{eq:lower-bound}
and the inequality in \eqref{eq:lower-bound-violation}.
Since $A_n\to\infty$, $k$ will go indefinitely by (ii).

First we claim that
\[v_n=o(b_n^2),\ w_n=o(h(b_n))\mbox{ and }\sum_{i=1}^nw_i=o(h(b_n)).\]
The second formula follows from $\sum_n w_n/h(b_n)<\infty$
and the third formula follows from $\sum_n w_n/h(b_n)<\infty$
and Kronecker's lemma.
We show that 
\[v_n=o(b_n^2).\]
Suppose otherwise.
Then, for some $c$ such that $0<c<1$,
\[\frac{\sqrt{v_n}}{b_n}>c\]
for infinitely many $n$.
Since $h(cx)/h(x)\ge c^3$,
\[\frac{h(\sqrt{v_n})}{h(b_n)}\ge \frac{h(cb_n)}{h(b_n)}\ge c^3\]
for infinitely many $n$, which contradicts the fact that
\[h(\sqrt{v_n})\le w_n=o(h(b_n))\]
by Proposition \ref{prop:vw}.

We claim that (i) and (iii) happen only finitely many times.
Consider the case that $k$ is sufficiently large.
Then $n$ is large, thus,
by the fact showed above,
we have
\begin{align}\label{eq:all-small}
v_n\le\frac{\delta^2}{2} b_n^2,\ w_n\le\frac{\delta}{2}h(b_n),\ 
\sum_{i=1}^n w_i\le\frac{\delta}{2} h(b_n)\mbox{ and }
|x_n|\le\frac{\delta}{2}\sqrt{\frac{A_n}{\ln\ln A_n}}.
\end{align}
If $A_n\ge C$, then $A_{n-1}<C$.
Then, in any case,
\[A_n=A_{n-1}+v_n<C+\frac{\delta^2}{2}\frac{A_n}{\ln\ln A_n}<C+\delta A_n,\]
which implies
\[C>(1-\delta)A_n.\]
Since $A_n$ is sufficiently large too,
\[\frac{b_n}{2}=\frac{1}{2}\sqrt{\frac{A_n}{\ln\ln A_n}}
<\sqrt{\frac{(1-\delta)A_n}{\ln\ln (1-\delta)A_n}},\]
thus, by \eqref{eq:all-small}, we have
\[v_n\le\frac{\delta^2C}{\ln\ln C},\ w_n\le\delta h(\sqrt{\frac{C}{\ln\ln C}}),\ 
\sum_{i=1}^n w_i\le\delta h(\sqrt{\frac{C}{\ln\ln C}})\]
and
\[|x_n|\le\delta\sqrt{\frac{C}{\ln\ln C}}.\]
Hence (i) and (iii) do not happen when $k$ is sufficiently large.

Note that $k$ is set to be $k+1$ at all but finitely many times.
As we showed above, we have
\[D^k=C>(1-\delta)A_n,\]
thus
\[D^{k+1}>(1-\delta)DA_n>A_n.\]


Hence from some $k$ on (ii) always happens and
\[\sum_{i=1}^n x_i\le(1-\epsilon)\sqrt{2C\ln\ln C}\]
will be satisfied.
Then $\limsup_n\cK_n=\infty$ because
\[\prod_k\left(1+\frac{1}{\ln D^k}\right)
=\prod_k\left(1+\frac{1}{k\ln D}\right)
=\infty.\]

This completes the proof of Proposition \ref{th:lower-bound}.

\subsection{Proof of the corollary}\label{subsec:proof-cor}

Finally we give a proof of Corollary \ref{cor:HW-S}.
First we give the definition of compliance.

\begin{definition}[Miyabe and Takemura \cite{Miyabe_convrandseries}]
\label{def:comply}
By a strategy $\cR$,
Reality \emph{complies} with the event $E$ if
\begin{enumerate}
\item irrespective of the moves of Forecaster and Skeptic,
both observing their collateral duties, $E$ happens, and
\item $\sup_n \cK_n<\infty$.
\end{enumerate}
\end{definition}

\begin{theorem}[Miyabe and Takemura \cite{Miyabe_convrandseries}]
In the unbounded forecasting,
if Skeptic can force an event $E$,
then Reality complies with $E$.
\end{theorem}

This theorem also holds for UFQSH
by essentially the same proof.

\begin{proof}[Proof of Corollary \ref{cor:HW-S}]
The implication of (i)$\Rightarrow$(ii) immediately follows from the main result.
The implication of (ii)$\Rightarrow$(iii) follows from the result above.

\bigskip
Let us show (iii)$\Rightarrow$(i).
Consider the case that Skeptic uses the strategy with which he can force
\begin{align}\label{eq:HW-LIL-no'}
\limsup_{n\to\infty}\frac{S_n-mn}{\sqrt{2n\ln\ln n}}=\sqrt{v},
\end{align}
and that Reality uses the strategy with which she can comply with
\eqref{eq:HW-LIL}.
Then both \eqref{eq:HW-LIL} and \eqref{eq:HW-LIL-no'} hold
for the realized path $\{x_n\}$.
This implies (i).
\end{proof}

\section*{Discussion}

We gave a sufficient condition for the law of the iterated logarithm
in game-theoretic probability with quadratic and stronger hedges.
The main difference from the result in Shafer and Vovk \cite{ShaVov01}
is that we could show the lower bound (sharpness)
in a non-predictable protocol.
The assumption of the stronger hedge is strong enough
to imply the result which has a similar form
as Hartman-Wintner's LIL and Strassen's converse.

However the condition \eqref{eq:cond-h} says that
there should be a gap between quadratic hedge
and the stronger hedge.
The authors do not know whether the condition can be weakened
so that the hedge is as close to quadratic one as one wants.
The authors also would like to know other formulations of i.i.d.\
in game-theoretic probability.

\section*{Acknowledgement}

The authors thank Shin-ichiro Takazawa
and the anonymous reviewer for their comments.
The first author was partially supported by GCOE, Kyoto University
and JSPS KAKENHI 23740072,
and the second author by the Aihara Project, the FIRST
program from JSPS, initiated by CSTP.

\bibliographystyle{abbrv}
\bibliography{lil}

\end{document}